\documentclass{article}
\usepackage{amsmath,amssymb,amsthm,ascmac,latexsym,mathrsfs,subfigure}
\usepackage[dvips]{graphicx}

\DeclareMathOperator{\N}{\mathbb{N}}
\DeclareMathOperator{\Z}{\mathbb{Z}}

\DeclareMathOperator{\R}{\mathbb{R}}

\providecommand{\abs}[1]{\lvert#1\rvert}
\providecommand{\Abs}[1]{\left\lvert#1\right\rvert}

\providecommand{\norm}[1]{\lVert#1\rVert}
\providecommand{\Norm}[1]{\left\lVert#1\right\rVert}

\newcommand{\mua}{\mu_\text{a}}
\newcommand{\mus}{\mu_\text{s}}

\theoremstyle{definition}
\newtheorem{dfn}{Definition}[section]
\newtheorem{thm}[dfn]{Theorem}
\newtheorem{lemma}[dfn]{Lemma}

\newtheorem{remark}[dfn]{Remark}
\newtheorem{example}[dfn]{Example}

\title{Stability and Convergence of an Upwind Finite Difference Scheme for the Radiative Transport Equation}

\author{Nobuyuki Higashimori$^{1}$ and Hiroshi Fujiwara$^{2}$ \\
($^{1}$ Graduate School of Economics, Hitotsubashi University) \\
($^{2}$ Graduate School of Informatics, Kyoto University)}

\date{}

\begin{document}

\maketitle

\begin{abstract}
An explicit numerical scheme is proposed for solving the 
initial-boundary value problem for the radiative transport equation in a 
rectangular domain with completely absorbing boundary condition. An 
upwind finite difference approximation is applied to the differential 
terms of the equation, and the composite trapezoidal rule to the 
scattering integral. The main results are positivity, stability, and 
convergence of the scheme. It is also shown that the scheme can be 
regarded as an iterative method for finding numerical solutions to the 
stationary transport equation. Some numerical examples for the 
two-dimensional problems are given. 
\end{abstract}

\section{Introduction}
\label{sect:1}

The radiative transport equation (RTE) is an integro-differential 
equation which is widely used as a mathematical model of the flow of 
particles traveling through a medium which absorbs and scatters the 
particles \cite{Dautray and Lions}. For example, let us regard the propagation of light in a 
turbid medium as a flow of photons and let $I=I(t,x,\xi)$ denote the 
density of the photons moving in the direction $\xi$ at time $t$ and 
position $x$. With suitable assumptions on the optical property of the 
medium, we can obtain an equation of the form \eqref{2.1a} (given 
below) which governs the time evolution of $I$. 

Let $\Omega$ be a rectangular domain with edges parallel to the coordinate axes 
in the $d$-dimensional Euclidean space $\R^d$ ($d=2$ or $3$) and let $S^{d-1}$ denote 
the unit sphere in $\R^d$. We consider the initial-boundary value problem 
\begin{subequations}
\label{2.1}
\begin{align}
\dfrac{1}{c(x)}\dfrac{\partial I}{\partial t}=-\xi\cdot\nabla_x I-\{\mua(x)+\mus(x)\}I
&+\mus(x)\int_{S^{d-1}}p(x;\xi,\xi')I(t,x,\xi')\,d\omega_{\xi'}+q(t,x,\xi),\,\ \nonumber \\
&\hspace{3cm}\text{$0<t<T,\ (x,\xi)\in \Omega\times S^{d-1}$},\label{2.1a} \\
I(0,x,\xi)&=I_0(x,\xi),\hspace{1.3cm}\text{$(x,\xi)\in \Omega\times S^{d-1}$},\label{2.1b} \\
I(t,x,\xi)&=I_1(t,x,\xi),\hspace{1cm}\text{$0<t<T,\ (x,\xi)\in\Gamma_{-}$},\label{2.1c}
\end{align}
\end{subequations}
where 
\[
\Gamma_{-}:=\left\{(x,\xi)\in \partial\Omega\times S^{d-1};\ n(x)\cdot\xi<0\right\},
\]
with $n(x)$ the outward unit normal vector at $x\in\partial\Omega$. In 
equation \eqref{2.1a} the dot $\cdot$ indicates the standard inner 
product in $\R^d$, $\nabla_x$ denotes the spatial gradient and 
$d\omega_{\xi'}$ the surface element on $S^{d-1}$. 
The coefficient $c(x)$ denotes the speed of particles at position $x$, $\mua(x)$ the absorption 
coefficient, and $\mus(x)$ the scattering coefficient. The integral 
kernel $p(x;\xi,\xi')$ is called the scattering phase function and 
expresses the density of the conditional probability of the event that a 
particle moving in the direction $\xi'$ is scattered into the direction 
$\xi$ at position $x$. The inhomogeneous term $q(t,x,\xi)$ is the 
density of internal source of particles emitted in the direction $\xi\in 
S^{d-1}$ at time $t$ and position $x$. 

An application of the RTE of the form \eqref{2.1a} can be found in 
researches on diffuse optical tomography (DOT), which is expected to be 
a new modality of noninvasive medical imaging which uses low-energy light \cite{Arridge,ArridgeSchotland}. 
The propagation of light in the biological tissue under examination is 
modeled by equation \eqref{2.1a} whose coefficients are 
unknown quantities to be determined by DOT. The coefficients represent optical 
properties of the biological tissue, and they actually depend not only on position $x$ 
but also on time $t$. However, the speed of light is so faster than time scale of 
tissue activity that we neglect their dependence on $t$. The procedure 
of imaging by DOT is regarded as an inverse problem of recovering the unknown 
coefficients by observing the solution $I$ at the 
boundary of the domain. When we analyze the inverse problem, it is extremely valuable to 
have an efficient numerical method for solving the problem \eqref{2.1} 
since it is practically impossible to have a simple analytic expression 
for the solution of \eqref{2.1}. 

The purpose of the present paper is to construct an 
explicit numerical scheme for solving the initial-boundary value problem 
\eqref{2.1}. In our discretization of \eqref{2.1a}, the 
differential operator $c^{-1}\partial_t+\xi\cdot\nabla_x$ is 
approximated by an upwind finite difference, and the scattering integral 
by the composite trapezoidal rule. 
We describe the detail only for the two-dimensional case in \S \ref{sect:2}, 
and give a brief account of the three-dimensional case in \S \ref{sect:3D}. 
The main results are the stability 
(Theorem \ref{thm:3.1}, Theorem \ref{thm:maximum}) 
and the convergence (Theorem \ref{thm:3.3}, Theorem \ref{thm:conv}) of the scheme. 
In \S \ref{sect:stat} we also show that the scheme can be regarded as 
an iterative method for finding numerical solutions to the stationary 
transport equation. 
In \S \ref{sect:num} we give examples of numerical solution of the initial-boundary value 
problem \eqref{2.1} and the corresponding stationary problem in the two-dimensional case. 

Finally we formulate our assumptions on the data of the problem 
\eqref{2.1} which are fixed throughout the paper. 
We assume that the coefficients $c$, $\mua$, and $\mus$ are measurable functions satisfying 
\begin{gather*}
0<c(x),\quad 
0\leq\mua(x),\quad 
0\leq\mus(x),\quad 
\mus(x)\not\equiv 0, \\
c^+:=\sup_{x\in\Omega}c(x)<+\infty,\qquad 
\mu^*:=\inf_{\{x;\ \mus(x)\neq 0\}}\frac{\mua(x)}{\mus(x)}>0. 
\end{gather*}
By its physical interpretation, the scattering kernel $p(x;\xi,\xi')$ must be nonnegative and satisfy 
\begin{equation}\label{normalize}
\int_{S^{d-1}}p(x;\xi,\xi')\,d\omega_{\xi}=1,\quad x\in\Omega,\ \xi'\in S^{d-1}.
\end{equation}
We assume moreover that $p$ depends on $\xi$ and $\xi'$ through the angle, say 
$\alpha$, formed by them, so that we can write 
\begin{equation}
\label{pp}
p(x;\xi,\xi')=\tilde{p}(x;\alpha), 
\end{equation}
with a nonnegative function $\tilde{p}(x;\alpha)$ 
which is $2\pi$-periodic and even in $\alpha$ for each $x$; 
we also call $\tilde{p}(x;\alpha)$ the phase function. 
Consequently, the condition \eqref{normalize} is equivalent to 
\[
\int_{S^{d-1}}p(x;\xi,\xi')\,d\omega_{\xi'}=1,\quad x\in\Omega,\ \xi\in S^{d-1},
\]
and is further rewritten in terms of $\tilde{p}$ as 
\begin{align*}
\int_{0}^{2\pi}\tilde{p}(x;\alpha)\,d\alpha=1\quad\text{for $d=2$}; \qquad 
2\pi\int_{0}^{\pi}\tilde{p}(x;\alpha)\sin\alpha\,d\alpha=1\quad\text{for $d=3$}. 
\end{align*}
We assume 
that $\tilde{p}(x;\alpha)$ is of class $C^2$ in $\alpha$ for each $x\in{\Omega}$ and that 
\[
\Norm{\frac{\partial^2\tilde{p}}{\partial \alpha^2}}_\infty
:=\sup_{(x,\alpha)\in\Omega\times[0,2\pi]}\,
\Abs{\frac{\partial^2\tilde{p}}{\partial \alpha^2}(x;\alpha)}<\infty.
\]
The inputs are assumed to be real bounded measurable functions: 
\[
q(t,x,\xi)\in L^{\infty}\bigl((0,T)\times\Omega\times S^{d-1}\bigr),\quad 
I_0(x,\xi)\in L^\infty(\Omega\times S^{d-1}),\quad 
I_1(t,x,\xi)\in L^\infty\bigl((0,T)\times\Gamma_{-}\bigr). 
\]
Their $L^\infty$-norms are denoted by $\norm{q}_\infty$, $\norm{I_0}_\infty$, and $\norm{I_1}_\infty$. 

\section{Upwind finite difference scheme for RTE and its properties for the two-dimensional case}
\label{sect:2}

Let $\Omega=(0,L_1)\times(0,L_2)$ be a rectangle in $\R^2$. 
In this section and the next we put $X=\Omega\times S^1$, which is the phase space of 
equation \eqref{2.1a}. 
For three positive integers $M_1$, $M_2$, and $M$ and a positive number $\Delta t$, 
we put 
\[
\Delta x_1=L_1/M_1,\quad \Delta x_2=L_2/M_2,\quad \Delta\theta=2\pi/M, 
\]
and consider the grid points $(t_k,x_{ij},\xi_n)$ defined by 
\begin{align*}
t_k=k\Delta t,\quad x_{ij}=(i\Delta x_1,j\Delta x_2),\quad 
\xi_n=(\xi_{n,1},\,\xi_{n,2})=(\cos n\Delta\theta,\sin n\Delta\theta),
\end{align*}
where $k,i,j,n\in\Z$. 
Using $I^k_{i,j,n}$ as the value corresponding to $I(t_k,x_{ij},\xi_n)$, 
we discretize the problem \eqref{2.1} as follows: 
\begin{subequations}\label{2.2}
\begin{align}
\dfrac{1}{c(x_{ij})}\dfrac{I^{k+1}_{i,j,n}-I^k_{i,j,n}}{\Delta t}
&=A_\Delta I^k_{i,j,n} - \Sigma_\Delta I^{k+1}_{i,j,n} + K_\Delta I^k_{i,j,n}+q^k_{i,j,n},\nonumber \\
&\hspace{3.4cm}\ 0\leq k\leq T/\Delta t-1,\ (x_{ij},\xi_n)\in X,\label{2.2a} \\
I^0_{i,j,n}&=I_0(x_{ij},\xi_n),\hspace{1.2cm}\ (x_{ij},\xi_n)\in X,\label{2.2b} \\
I^k_{i,j,n}&=I_1(t_k,x_{ij},\xi_n),\hspace{0.8cm}\ 0\leq k\leq T/\Delta t,\,\ (x_{ij},\xi_n)\in\Gamma_{-},
\label{2.2c}
\end{align}
\end{subequations}
where $q^k_{i,j,n}=q(t_k,x_{ij},\xi_n)$, and the operators $A_\Delta$, 
$\Sigma_\Delta$, and $K_\Delta$ are defined by 
\begin{subequations}
\begin{align}
A_\Delta I^k_{i,j,n}
&=-\xi_{n,1}\dfrac{I^k_{i+1,j,n}-I^k_{i-1,j,n}}{2\Delta x_1}
+\abs{\xi_{n,1}}\dfrac{I^k_{i+1,j,n}-2I^k_{i,j,n}+I^k_{i-1,j,n}}{2\Delta x_1}\nonumber \\
&\quad -\xi_{n,2}\dfrac{I^k_{i,j+1,n}-I^k_{i,j-1,n}}{2\Delta x_2}
+\abs{\xi_{n,2}}\dfrac{I^k_{i,j+1,n}-2I^k_{i,j,n}+I^k_{i,j-1,n}}{2\Delta x_2},\label{a-delta} \\
\Sigma_\Delta I^k_{i,j,n}&=\{\mus(x_{ij})+\mua(x_{ij})\}I^k_{i,j,n},\label{s-delta} \\
K_\Delta I^k_{i,j,n}&=\mus(x_{ij})\Delta\theta\sum_{\nu=0}^{M-1}p(x_{ij};\xi_n,\xi_\nu)I^k_{i,j,\nu}.\label{k-delta}
\end{align}
\end{subequations}

\begin{remark}
(i) The operator $A_\Delta$ is defined so that the combination 
\begin{equation}
\label{upwind}
\dfrac{1}{c(x_{ij})}\dfrac{I^{k+1}_{i,j,n}-I^k_{i,j,n}}{\Delta t}-A_\Delta I^k_{i,j,n}
\end{equation}
gives the first-order upwind finite difference approximation to the partial differential operator 
${c}^{-1}\partial_t+\xi\cdot\nabla_x$. 
For example, 
for $n$ with $\xi_{n,1}>0$ and $\xi_{n,2}\geq 0$, 
\eqref{upwind} is equal to 
\[
\dfrac{1}{c(x_{ij})}\dfrac{I^{k+1}_{i,j,n}-I^k_{i,j,n}}{\Delta t}
+\xi_{n,1}\dfrac{I^k_{i,j,n}-I^k_{i-1,j,n}}{\Delta x_1}
+\xi_{n,2}\dfrac{I^k_{i,j,n}-I^k_{i,j-1,n}}{\Delta x_2}. 
\]
The other three cases can be checked similarly. 

(ii) Scheme \eqref{2.2} is explicit; see \eqref{re-2.2a}. 
\end{remark}

\begin{thm}
\label{thm:3.1}
{\bf (Stability and positivity for $\boldsymbol{d=2}$)} 
Suppose that the discretization parameters satisfy 
\begin{subequations}
\label{3.1}
\begin{gather}
\dfrac{c^+\Delta t}{\Delta x_1}+\dfrac{c^+\Delta t}{\Delta x_2}\leq 1,\label{3.1a}\\
\Norm{\frac{\partial^2\tilde{p}}{\partial\alpha^2}}_\infty
\Delta\theta^2\leq\dfrac{12}{\pi}\mu^*.\label{3.1b}
\end{gather}
\end{subequations}
Then we have 
\begin{equation*}
\norm{I^k}_\infty
\leq \norm{I_0}_\infty+\norm{I_1}_\infty+c^{+}\norm{q}_{\infty}T,\quad 0\leq k\leq T/\Delta t, 
\end{equation*}
where 
\[
\norm{I^k}_\infty:=\sup\left\{\abs{I^k_{i,j,n}};\ (x_{ij},\xi_n)\in X\cup\Gamma_{-}\right\}.
\]
Moreover if $q\geq 0$, $I_0\geq 0$, and $I_1\geq 0$ everywhere, then 
$I^k_{i,j,n}\geq 0$ holds for all $k,i,j,n$. 
\end{thm}

In the proofs below we write $c$, $\mua$, and $\mus$ instead of 
$c(x_{ij})$, $\mua(x_{ij})$, and $\mus(x_{ij})$, respectively. 

\begin{proof}
We show that 
\begin{equation}
\label{induction}
\norm{I^k}_\infty
\leq\norm{I_0}_\infty+\norm{I_1}_\infty+c^{+}\norm{q}_{\infty}t_k 
\end{equation}
for $0\leq k\leq T/\Delta t$ by induction on $k$. It is obvious for $k=0$. 
Assume now that \eqref{induction} holds for some $k\leq T/\Delta t-1$. 
The proof shall be finished by showing that 
\begin{equation}
\label{key}
\Abs{I^{k+1}_{i,j,n}}\leq\norm{I^k}_\infty+c^+\norm{q}_{\infty}\Delta t, 
\qquad (x_{ij},\xi_n)\in X. 
\end{equation}
Indeed, it follows from \eqref{induction} and \eqref{key} that 
\[
\Abs{I^{k+1}_{i,j,n}}\leq
\norm{I_0}_\infty+\norm{I_1}_\infty+c^+\norm{q}_{\infty}t_{k+1}, \quad (x_{ij},\xi_n)\in X. 
\]
Since $\Abs{I^{k+1}_{i,j,n}}=\abs{I_1(t_k,x_{ij},\xi_n)}\leq\norm{I_1}_\infty$ 
for $(x_{ij},\xi_n)\in\Gamma_{-}$, we obtain 
\[
\norm{I^{k+1}}_\infty\leq
\norm{I_0}_\infty+\norm{I_1}_\infty+c^+\norm{q}_{\infty}t_{k+1}. 
\]

We now show \eqref{key}. We introduce an auxiliary operator $B_\Delta$ by the relation 
\begin{equation*}
A_\Delta I^k_{i,j,n}=\left(-\frac{\abs{\xi_{n,1}}}{\Delta x_1}-\frac{\abs{\xi_{n,2}}}{\Delta x_2}\right)
I^k_{i,j,n}+B_\Delta I^k_{i,j,n},\qquad (x_{ij},\xi_n)\in X,
\end{equation*}
or explicitly 
\begin{align}
B_\Delta I^k_{i,j,n}
&=\frac{(\abs{\xi_{n,1}}-\xi_{n,1})I^k_{i+1,j,n}
  +(\abs{\xi_{n,1}}+\xi_{n,1})I^k_{i-1,j,n}}{2\Delta x_1}\nonumber\\
&\qquad+\frac{(\abs{\xi_{n,2}}-\xi_{n,2})I^k_{i,j+1,n}+(\abs{\xi_{n,2}}+\xi_{n,2})I^k_{i,j-1,n}}{2\Delta x_2}. 
\label{b-delta}
\end{align}
Then we write \eqref{2.2a} as 
\begin{align}
\{1&+c\Delta t(\mus+\mua)\}I^{k+1}_{i,j,n}\nonumber\\
&=\left(1-\abs{\xi_{n,1}}\frac{c\Delta t}{\Delta x_1}-\abs{\xi_{n,2}}\frac{c\Delta t}{\Delta x_2}\right)I^k_{i,j,n}
+c\Delta t\left(B_\Delta I^k_{i,j,n}+K_\Delta I^k_{i,j,n}+q^k_{i,j,n}\right).
\label{re-2.2a}
\end{align}
We find from \eqref{3.1a} that 
\begin{equation*}
1-\abs{\xi_{n,1}}\frac{c\Delta t}{\Delta x_1}-\abs{\xi_{n,2}}\frac{c\Delta t}{\Delta x_2}\geq 0,
\end{equation*}
and from \eqref{b-delta} that 
\begin{equation*}
\Abs{B_\Delta I^k_{i,j,n}}\leq
\left(\frac{\abs{\xi_{n,1}}}{\Delta x_1}+\frac{\abs{\xi_{n,2}}}{\Delta x_2}\right)\norm{I^k}_\infty, 
\end{equation*}
hence from \eqref{re-2.2a} that 
\begin{align}
\{1&+c\Delta t(\mus+\mua)\}\Abs{I^{k+1}_{i,j,n}}\nonumber\\
&\leq
\left(1-\abs{\xi_{n,1}}\frac{c\Delta t}{\Delta x_1}-\abs{\xi_{n,2}}\frac{c\Delta t}{\Delta x_2}\right)
\norm{I^k}_{\infty}+c\Delta t\left(\Abs{B_\Delta I^k_{i,j,n}}+\Abs{K_\Delta I^k_{i,j,n}}
+\norm{q}_{\infty}\right) \nonumber\\
&
\leq\norm{I^k}_\infty+c\Delta t\Abs{K_\Delta I^k_{i,j,n}}+c\Delta t\norm{q}_{\infty}.\label{yyy}
\end{align}
By \eqref{k-delta} and \eqref{pp} we have 
\begin{align}
\Abs{K_\Delta I^k_{i,j,n}}
\leq\mus\Delta\theta\sum_{\nu=0}^{M-1}p(x_{ij};\xi_n,\xi_\nu)\norm{I^k}_\infty
=\norm{I^k}_\infty\,\mus\Delta\theta\sum_{\nu=0}^{M-1}\tilde{p}(x_{ij};\theta_\nu-\theta_n).
\end{align}
Here we use the next lemma. 

\begin{lemma}\label{lem:trapezoid}
Let $f(\theta)$ be a $2\pi$-periodic function of class $C^2$. 
Let $M$ be a positive integer and 
\[
\Delta\theta=\frac{2\pi}{M},\quad \theta_m=m\Delta\theta,\quad m=0,1,\dotsc,M-1.
\]
Then there exists a point $\theta'\in[0,2\pi)$ such that 
\[
\int_{0}^{2\pi}f(\theta)\,d\theta-\Delta\theta\sum_{m=0}^{M-1}f(\theta_m)
=\frac{\pi}{12}{f''(\theta')}\Delta\theta^2.
\]
\end{lemma}

This can be shown routinely by taking the periodicity of $f$ into account. 

Applying Lemma~\ref{lem:trapezoid} to the function $f(\theta)=\tilde{p}(x_{ij};\theta-\theta_n)$, 
we find 
\begin{align*}
\int_{0}^{2\pi}\tilde{p}(x_{ij};\theta-\theta_n)\,d\theta
-\Delta\theta\sum_{\nu=0}^{M-1}\tilde{p}(x_{ij};\theta_\nu-\theta_n)
&=\frac{\pi}{12}\Delta\theta^2
\left.\dfrac{\partial^2}{\partial\theta^2}\tilde{p}(x_{ij};\theta-\theta_n)\right\rvert_{\theta=\theta'} 
\end{align*}
for some $\theta'={\theta'}_{i,j,n}$. 
Since the integral of $\tilde{p}$ is equal to $1$, 
we get 
\begin{equation}
\label{zzz}
\Delta\theta\sum_{\nu=0}^{M-1}\tilde{p}(x_{ij};\theta_\nu-\theta_n)
=1-\frac{\pi}{12}\Delta\theta^2
\left.\dfrac{\partial^2}{\partial\theta^2}\tilde{p}(x_{ij};\theta-\theta_n)\right\rvert_{\theta=\theta'}
\leq 1+\mu^* 
\end{equation}
by using \eqref{3.1b}. 
Combining \eqref{yyy}--\eqref{zzz} and observing that $\mus\mu^*\leq\mua$ on $\Omega$, we obtain 
\[
\{1+c\Delta t(\mus+\mua)\}\Abs{I^{k+1}_{i,j,n}}
\leq\{1+c\Delta t(\mus+\mua)\}\norm{I^k}_\infty+c\norm{q}_{\infty}\Delta t. 
\]
Dividing both sides by $1+c\Delta t(\mus+\mua)$, we arrive at \eqref{key}.

The positivity is shown as follows: 
if $I^k_{i,j,n}\geq 0$ for all $i,j,n$ at a time level $k$, 
then we have 
$K_\Delta I^k_{i,j,n}\geq 0$ by \eqref{k-delta}, and $B_\Delta I^k_{i,j,n}\geq 0$ by \eqref{b-delta}. 
Hence $I^{k+1}_{i,j,n}\geq 0$ at the next time level $k+1$ by \eqref{re-2.2a} and $q\geq 0$. 
\end{proof}

\begin{thm}
\label{thm:3.3}
{\bf (Convergence for $\boldsymbol{d=2}$)} 
In addition to the hypotheses of Theorem \ref{thm:3.1}, 
we suppose the following three conditions: 
(i) 
the mesh ratios $\Delta t/\Delta x_1$ and $\Delta t/\Delta x_2$ are kept fixed; 
(ii) 
there exist positive numbers $\mua^+$ and $\mus^+$ such that 
\[
0\leq\mua(x)\leq\mua^+,\quad 0\leq\mus(x)\leq\mus^+,\qquad x\in\Omega;
\]
and (iii) 
the solution $I(t,x,\xi)$ of \eqref{2.1} belongs to 
$C^2\bigl([0,T)\times(X\cup\Gamma_{-})\bigr)$ and is bounded together with its 
derivatives of the first and the second order. 
Then we have 
\begin{equation*}
\|I(t_k,\cdot,\cdot)-I^k\|_\infty\leq C(\Delta t + \Delta\theta^2), \quad 0\leq k\leq T/\Delta t,
\end{equation*}
where $C$ is a positive number independent of $\Delta t$ and $\Delta\theta$.
\end{thm}

\begin{proof}
Let $I(t,x,\xi)$ be the solution of \eqref{2.1} satisfying (iii). 
We define the local truncation error $\tau^k_{i,j,n}$ 
by the relation 
\begin{align}
&\frac{1}{c}\frac{I(t_{k+1},x_{ij},\xi_n)-I(t_k,x_{ij},\xi_n)}{\Delta t}\nonumber\\
&\qquad=A_\Delta I(t_k,x_{ij},\xi_n)-\Sigma_\Delta I(t_{k+1},x_{ij},\xi_n)+K_\Delta I(t_k,x_{ij},\xi_n)
+q^k_{i,j,n}+\tau^k_{i,j,n},
\label{3.2}
\end{align}
where the operators $A_\Delta$, $\Sigma_\Delta$, and $K_\Delta$ 
are applied to $I(t_k,x_{ij},\xi_n)$ in place of $I^k_{i,j,n}$.
By using equation \eqref{2.1a}, Taylor's theorem, and 
by applying Lemma \ref{lem:trapezoid} to the function 
\[
\text{$f(\theta)=p(x_{ij};\xi_n,\eta)I(t_k,x_{ij},\eta)$ with $\eta=(\cos\theta,\sin\theta)$}, 
\]
we arrive at 
\begin{align*}
\tau^k_{i,j,n}
&=\frac{1}{2c}\frac{\partial^2 I}{\partial t^2}(t_k',x_{ij},\xi_n)\Delta t
  -\frac{\xi_{n,1}}{2}\frac{\partial^2 I}{\partial {x_1}^2}(t_k,x'_{ij},\xi_n)\Delta x_1
  -\frac{\xi_{n,2}}{2}\frac{\partial^2 I}{\partial {x_2}^2}(t_k,x''_{ij},\xi_n)\Delta x_2\\
&\qquad+(\mus+\mua)\frac{\partial I}{\partial t}(t_k'',x_{ij},\xi_n)\Delta t
  +\frac{\pi\mus}{12}\Delta\theta^2
  \left.\frac{\partial^2}{\partial\theta^2}
  p(x_{ij};\xi_n,\eta)I(t_k,x_{ij},\eta)\right\vert_{\theta=\theta'}, 
\end{align*}
for some $t'_k$, $t''_k$, $x'_{ij}$, and $x''_{ij}$. 
Multiplying this equality by $c$ and using the assumptions, we obtain 
\begin{equation}
\label{eee}
c^+\sup_{k,i,j,n}\abs{\tau^k_{i,j,n}}\leq C'(\Delta t+\Delta\theta^2), 
\end{equation}
with $C'$ independent of $\Delta t$ and $\Delta\theta$. 

Next we consider the discretization error $E^k_{i,j,n}:=I^k_{i,j,n}-I(t_k,x_{ij},\xi_n)$. 
We subtract \eqref{3.2} from \eqref{2.2a} to have 
\begin{align*}
\frac{1}{c}\frac{E^{k+1}_{i,j,n}-E^k_{i,j,n}}{\Delta t}
=A_\Delta E^k_{i,j,n}-\Sigma_\Delta E^{k+1}_{i,j,n}+K_\Delta E^k_{i,j,n}-\tau^k_{i,j,n},\qquad & \nonumber\\
0\leq k\leq T/\Delta t-1,\quad (x_{ij},\xi_n)\in X.& 
\end{align*}
We observe that $E=\{E^k_{i,j,n}\}$ satisfies 
equation \eqref{2.2a} with $q^k_{i,j,n}=-\tau^k_{i,j,n}$, 
and the homogeneous initial-boundary conditions. 
Therefore the proof of Theorem \ref{thm:3.1} shows that 
\[
\norm{E^{k+1}}_{\infty}\leq 0+0+Tc^+\sup_{k,i,j,n}\abs{\tau^k_{i,j,n}},
\qquad 0\leq k\leq T/\Delta t-1.
\]
Then \eqref{eee} gives 
\[
\norm{E^k}_\infty
\leq C(\Delta t+\Delta\theta^2),
\qquad 0\leq k\leq T/\Delta t,
\]
with $C:=TC'$. 
\end{proof}

\begin{remark}\label{rem:analytic}
Here we consider a particular situation: 
$\tilde{p}(x;\theta)$ is independent of $x$ and 
analytic in $\theta$ with period $2\pi$. We write $\tilde{p}(\theta)=\tilde{p}(x;\theta)$ 
and expand it into the Fourier series 
\begin{equation}\label{FT}
\tilde{p}(\theta)=\sum_{n=-\infty}^{\infty}c_n e^{in\theta},\qquad
c_n=\frac{1}{2\pi}\int_{0}^{2\pi}\tilde{p}(\theta)e^{-in\theta}\,d\theta.
\end{equation}
Since $\tilde{p}(\theta)$ is analytic, there exist numbers 
$C>0$ and $0<r<1$ such that (see \cite{Kress}) 
\begin{gather}
\abs{c_n}\leq Cr^{\abs{n}},\qquad n\in\Z,\nonumber\\
\Abs{\int_{0}^{2\pi}\tilde{p}(\theta)d\theta-\Delta\theta\sum_{\nu=0}^{M-1}\tilde{p}(\theta_m)}
\leq\frac{4\pi Cr^M}{1-r^M},\qquad M\in\N,\ \Delta\theta=\frac{2\pi}{M}.\label{A}
\end{gather}
Then the conclusions of Theorem \ref{thm:3.1} and Theorem \ref{thm:3.3} can be obtained if we assume 
\begin{equation}
\frac{4\pi Cr^M}{1-r^M}\leq\mu^*,\label{B}
\end{equation}
instead of \eqref{3.1b}. 
Indeed, since $\displaystyle\int_{0}^{2\pi}\tilde{p}(\theta)\,d\theta=1$, 
we find from \eqref{A} and \eqref{B} that 
\[
\Delta\theta\sum_{\nu=0}^{M-1}\tilde{p}(\theta_\nu-\theta_n)
\leq 1+\mu^*, 
\]
which corresponds to \eqref{zzz}. 
The proofs of the theorems are valid without further modification. 
\end{remark}

\begin{example}
In researches on diffuse optical tomography \cite{Arridge,ArridgeSchotland}, the 
main interest is in the three-dimensional case, and the Poisson kernel
\[
\tilde{p}(\theta)=\frac{1}{4\pi}\frac{1-g^2}{(1-2g\cos\theta+g^2)^{3/2}},\quad \text{with}\quad 
0\leq g<1,
\]
is frequently used as the phase function. 
The Poisson kernel is often referred to as the Henyey-Greenstein kernel in that field and 
the parameter $g$ indicates the strength of forward scattering by the medium \cite{Henyey}. 
If $g=0$, then $\tilde{p}(\theta)={1}/{4\pi}$ identically 
and the particles will be scattered uniformly in all directions. If $g$ gets 
closer to $1$, the proportion of 
photons scattered in directions near the incident one gets larger. 

Since we are dealing with the two-dimensional case in this section, 
we consider the two-dimensional analogue: 
\begin{equation*}
\tilde{p}(\theta)=\frac{1}{2\pi}\frac{1-g^2}{1-2g\cos\theta+g^2}\quad \text{with}\quad 
0\leq g<1. 
\end{equation*}
Its Fourier series expansion 
is 
\[
\tilde{p}(\theta)=\frac{1}{2\pi}\sum_{n=-\infty}^{\infty}g^{\abs{n}}e^{in\theta},
\quad\text{i.e., $\quad c_n=\frac{g^{\abs{n}}}{2\pi}$ in \eqref{FT},}
\]
so that, by Remark \ref{rem:analytic}, the scheme is stable if $\Delta x_1$, 
$\Delta x_2$, $\Delta t$, and $M$ satisfy \eqref{3.1a} and 
\begin{equation}\label{5.1}
\frac{2g^M}{1-g^M}\leq\mu^*.
\end{equation}
If $g=0$, in particular, the condition \eqref{5.1} trivially holds, so that 
the stability is independent of $M$. 
\end{example}

\section{Numerical solution of the stationary problem for RTE}
\label{sect:stat}

We turn our attention to the initial-boundary value problem \eqref{2.1} 
for an infinite interval of time $0<t<\infty$ with time-independent inputs 
$q=q(x,\xi)$ and $I_1=I_1(x,\xi)$. 
We consider the scheme \eqref{2.2} for all $k=0,1,2,\dotsc$ making an additional assumption that 
\[
(c\mua)^-:=\inf_{x\in\Omega}c(x)\mua(x)>0.
\]

Our aim here is to show that the scheme \eqref{2.2} serves an iterative method 
for approximating the solution of the corresponding stationary problem: 
\begin{subequations}\label{8.1}
\begin{gather}
-\xi\cdot\nabla_x J-\bigl(\mus(x)+\mua(x)\bigr)J+\mus(x)\int_{S^1}p(x;\xi,\xi')J(x,\xi')\,d\omega_{\xi'}
+q(x,\xi)=0,\quad (x,\xi)\in X,\label{8.1a}\\
J(x,\xi)=I_1(x,\xi),\qquad (x,\xi)\in\Gamma_{-}.\label{8.1b}
\end{gather}
\end{subequations}
We discretize the problem \eqref{8.1} as follows: 
\begin{subequations}\label{8.2}
\begin{align}
(A_\Delta-\Sigma_\Delta+K_\Delta)J_{i,j,n}+q_{i,j,n}=0,&\qquad(x_{ij},\xi_n)\in X,\label{8.2a}\\
J_{i,j,n}=I_1(x_{ij},\xi_n),&\qquad(x_{ij},\xi_n)\in\Gamma_{-},\label{8.2b}
\end{align}
\end{subequations}
with $J_{i,j,n}$ corresponding to $J(x_{ij},\xi_n)$. 
Then \eqref{8.2} is uniquely solvable for sufficiently small $\Delta\theta$ 
as the next theorem shows. 

\begin{thm}
\label{thm:solvable}
Suppose that $\Delta\theta$ satisfies 
\[
\norm{\tilde{p}}_\infty\Delta\theta\leq 1\quad\text{and}\quad
\Norm{\frac{\partial^2\tilde{p}}{\partial\theta^2}}_\infty\Delta\theta^2
<\frac{12}{\pi}\mu^*.
\]
Then there exists a unique solution $J_\Delta=\{J_{i,j,n}\}$ of \eqref{8.2} and it satisfies 
\begin{equation}\label{stat-stab}
\norm{J_\Delta}_\infty:=\sup_{i,j,n}\Abs{J_{i,j,n}}\leq\norm{I_1}_\infty+C\norm{q}_\infty,
\end{equation}
where $C$ is a positive number independent of $\Delta x_1$, $\Delta x_2$, and $\Delta\theta$. 
\end{thm}

We refer to \cite{Fujiwara} for a proof of Theorem \ref{thm:solvable} and 
instead show that the scheme \eqref{2.2} generates an iterative method 
for solving \eqref{8.2}.

\begin{thm}
\label{thm:4.1}
Suppose that the discretization parameters for the scheme \eqref{2.2} 
are fixed so that 
\begin{equation*}
\dfrac{c^+\Delta t}{\Delta x_1}+\dfrac{c^+\Delta t}{\Delta x_2}\leq 1,\quad
\norm{\tilde{p}}_\infty\Delta\theta\leq 1,\quad\text{and}\quad
\Norm{\frac{\partial^2\tilde{p}}{\partial\alpha^2}}_\infty\Delta\theta^2
<\dfrac{12}{\pi}\mu^*.
\end{equation*}
Then there exists a number $\rho$, $0<\rho<1$, such that 
\[
\norm{I^k-J_\Delta}_\infty\leq \left(\norm{I_0}_\infty+\norm{I_1}_\infty\right)\rho^k,\qquad
k\geq 0.
\]
\end{thm} 

\begin{proof}
We put $R^k_{i,j,n}=I^k_{i,j,n}-J_{i,j,n}$. 
Writing \eqref{8.2a} as 
\[
0=\frac{1}{c}\frac{J_{i,j,n}-J_{i,j,n}}{\Delta t}
=(A_\Delta-\Sigma_\Delta+K_\Delta)J_{i,j,n}+q_{i,j,n},\quad(x_{ij},\xi_n)\in X,
\]
and subtracting it from \eqref{2.2a}, we have 
\[
\frac{1}{c}\frac{R^{k+1}_{i,j,n}-R^k_{i,j,n}}{\Delta t}=A_\Delta R^k_{i,j,n}-\Sigma_\Delta R^{k+1}_{i,j,n}
+K_\Delta R^k_{i,j,n},\quad k\geq 0,\quad (x_{ij},\xi_n)\in X, 
\]
Let $\lambda$ be a number, $0<\lambda<1$, determined by 
\[
\Norm{\frac{\partial^2\tilde{p}}{\partial\alpha^2}}_\infty\Delta\theta^2
=\lambda\dfrac{12}{\pi}\mu^*. 
\]
Reasoning as in the proof of Theorem \ref{thm:3.1}, we obtain 
\begin{align*}
\{1+c\Delta t(\mus+\mua)\}\Abs{R^{k+1}_{i,j,n}}
&\leq\{1+c\Delta t(\mus+\lambda\mua)\}\norm{R^k}_\infty,\quad k\geq 0,\quad 
(x_{ij},\xi_n)\in X. 
\end{align*}
Since $R^{k+1}_{i,j,n}=0$ for $(x_{ij},\xi_n)\in\Gamma_{-}$, 
we have 
\[
\norm{R^{k+1}}_\infty
\leq\left(
\sup_{\Omega}
\frac{1+c\mus\Delta t+\lambda c\mua\Delta t}{1+c\mus\Delta t+\phantom{\lambda}c\mua\Delta t}
\right)
\norm{R^k}_\infty.
\]
Observing that $\mus\leq\mua/\mu^*$ and $c\mua\geq(c\mua)^-$ on $\Omega$, we see that 
\begin{align*}
\frac{1+c\mus\Delta t+\lambda c\mua\Delta t}{1+c\mus\Delta t+\phantom{\lambda}c\mua\Delta t}
&\leq
\frac{1+c\mua\Delta t/\mu^*+\lambda c\mua\Delta t}{1+c\mua\Delta t/\mu^*+\phantom{\lambda}c\mua\Delta t}
=
\frac{1+c\mua\Delta t(1/\mu^*+\lambda)}{1+c\mua\Delta t(1/\mu^*+1)}\\
&\leq
\frac{1+(c\mua)^-\Delta t(1/\mu^*+\lambda)}{1+(c\mua)^-\Delta t(1/\mu^*+1)}=:\rho<1.
\end{align*}
Thus we obtain $\norm{R^{k+1}}_{\infty}\leq\rho\norm{R^k}_{\infty}$ for all $k\geq 0$ and hence 
\[
\norm{I^k-J_\Delta}_\infty=\norm{R^k}_\infty
\leq\rho^k\norm{R^0}_\infty=\rho^k\norm{I^0-J_\Delta}_\infty,\quad k\geq 0. 
\]
The proof completes by using \eqref{stat-stab}. 
\end{proof}

\section{Numerical Examples}
\label{sect:num}

Here we consider a square domain $\Omega = (0,50)\times(0,50)\subset\R^2$ of sides $50$ mm 
and assume that the domain is occupied by a medium of uniform optical properties: 
$\mua = 0.08$ mm$^{-1}$, $\mus = 1.09$ mm$^{-1}$, and $c = 0.196$ mm/ps. 
Units of length and time in our numerical computation are chosen to be 
millimeter and picosecond, respectively. We employ the Poisson kernel 
\[
\tilde{p}(\theta)=\frac{1}{2\pi}\frac{1-g^2}{1-2g\cos\theta+g^2},\quad g=0.9
\]
as the phase function. The boundary condition is set to 
\[
I_1(t,x,\xi) =
\begin{cases}
\dfrac{1}{\sqrt{2\pi}\sigma}
    \exp\left(-\dfrac{(\theta-\pi/2)^2}{2\sigma^2}\right),
    \quad & 24.9 \leq x_1 \leq 25.1,\ x_2 = 0,\\
0, \quad & \text{otherwise},
\end{cases}
\]
where $\sigma = 0.2$ and $\xi = (\cos\theta,\sin\theta)$. The initial 
condition is $I_0(x,\xi) = 0$ in $X$, and $I_0(x,\xi)$ 
and $I_1(t,x,\xi)$ satisfy compatibility on $\Gamma_{-}$. The parameters 
and the inputs are chosen so as to simulate a phantom experiment carried 
out at Human Brain Research Center, Kyoto University \cite{HBRC}. 

Numerical computation is processed with $\Delta t = 0.1$, 
$\Delta x = \Delta y = 0.1$, and $\Delta \theta = 2\pi / 60$. 
These satisfy the stability conditions \eqref{3.1a} and \eqref{5.1}, since 
\[
60=M \ge \log_g \dfrac{\mua}{2\mus + \mua} \approx 31.71.
\]
Fig.~\ref{fig:xi} shows the numerical solution $I^k_{i,j,n}$ for several points $x_{ij}$. 
The solution is shown in the polar coordinate with respect to $\theta$ for $\xi = (\cos\theta,\sin\theta)$. 
\begin{figure}[tbp]
\subfigure[$t=50 \quad (k=500)$]{\includegraphics[width=.5\textwidth]{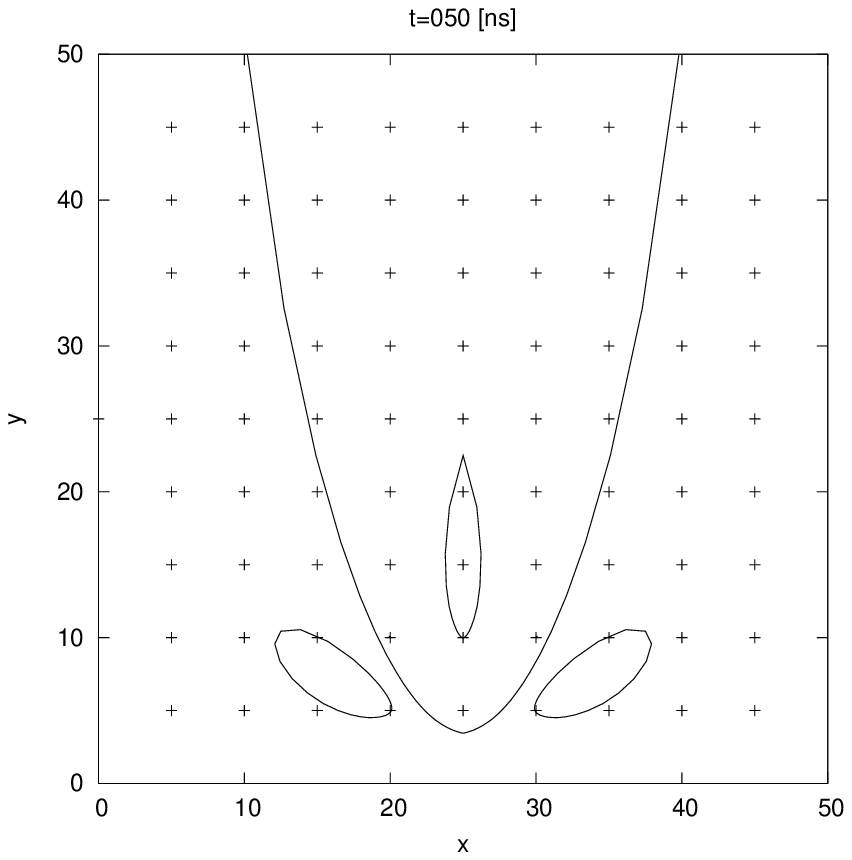}}
\subfigure[$t=100 \quad (k=1000)$]{\includegraphics[width=.5\textwidth]{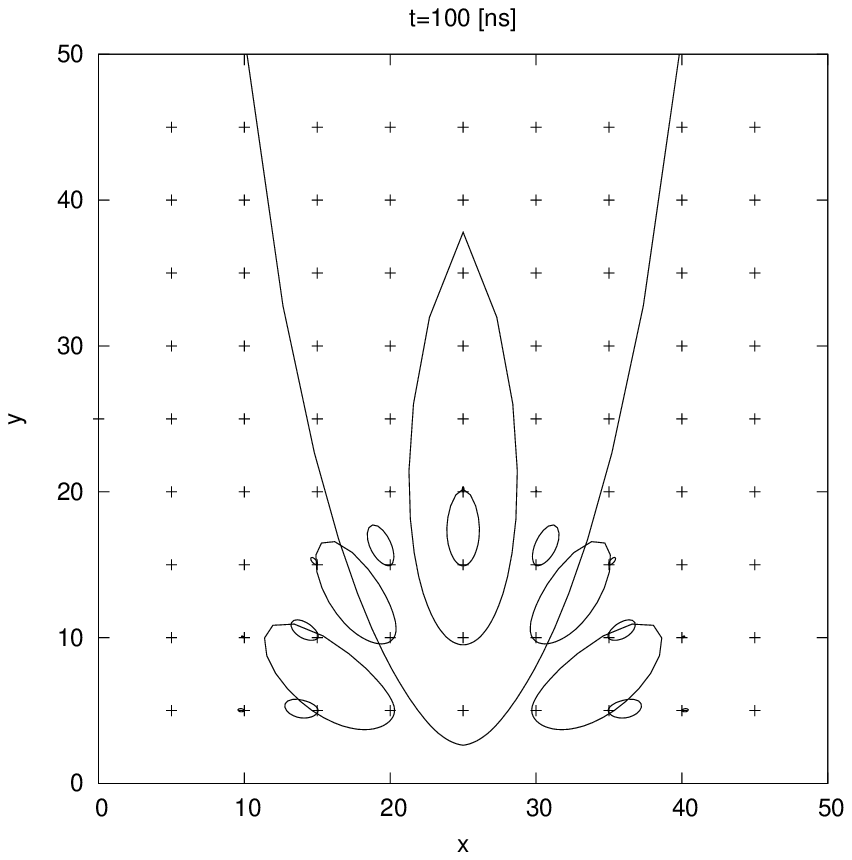}}
\subfigure[$t=200 \quad (k=2000)$]{\includegraphics[width=.5\textwidth]{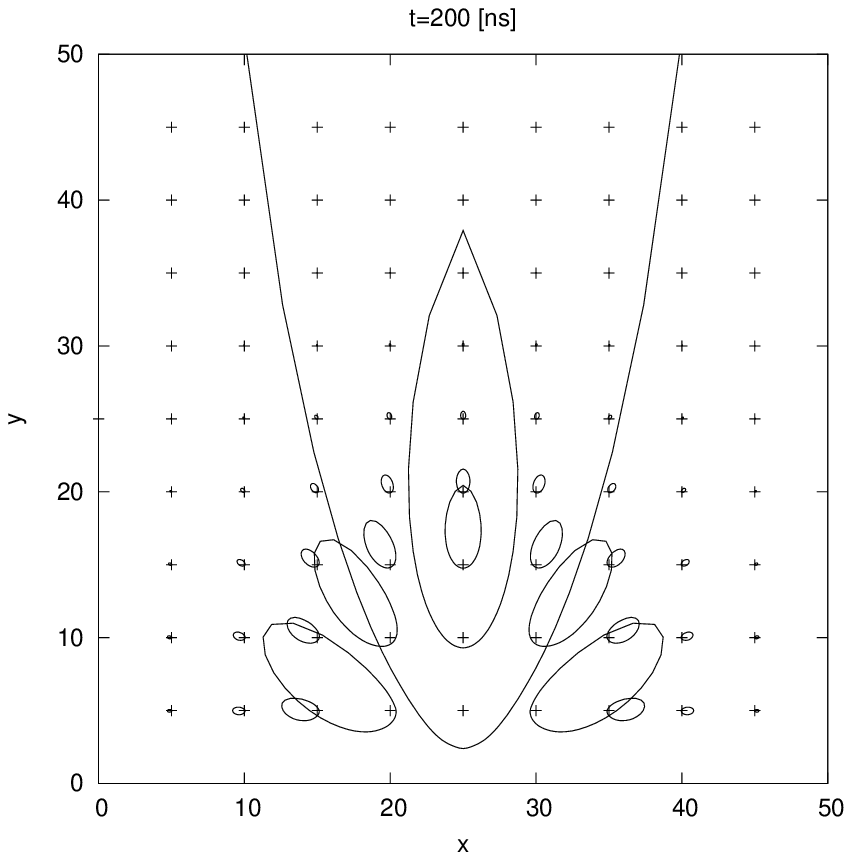}}
\subfigure[$t=400 \quad (k=4000)$]{\includegraphics[width=.5\textwidth]{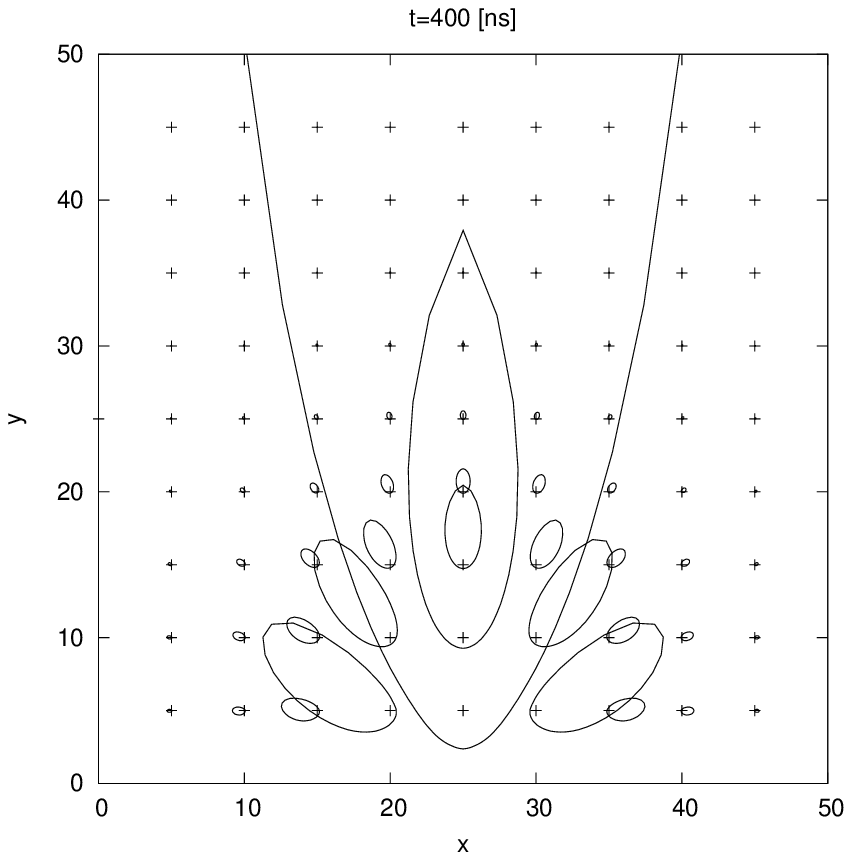}}
\caption{\label{fig:xi}Numerical Solutions $I^k_{i,j,n}$ 
in polar coordinate with respect to $\xi \in S^1$ at point $x_{ij}$. 
}
\end{figure}
Fig.~\ref{fig:contourlightintensity} shows contour lines of integrated 
light intensity $\displaystyle\int_{S^1} I(t,x,\xi) d\omega_{\xi}$ at 
time $t$ and position $x$. 
\begin{figure}[tbp]
\subfigure[$t=50 \quad (k=500)$]{\includegraphics[width=.5\textwidth]{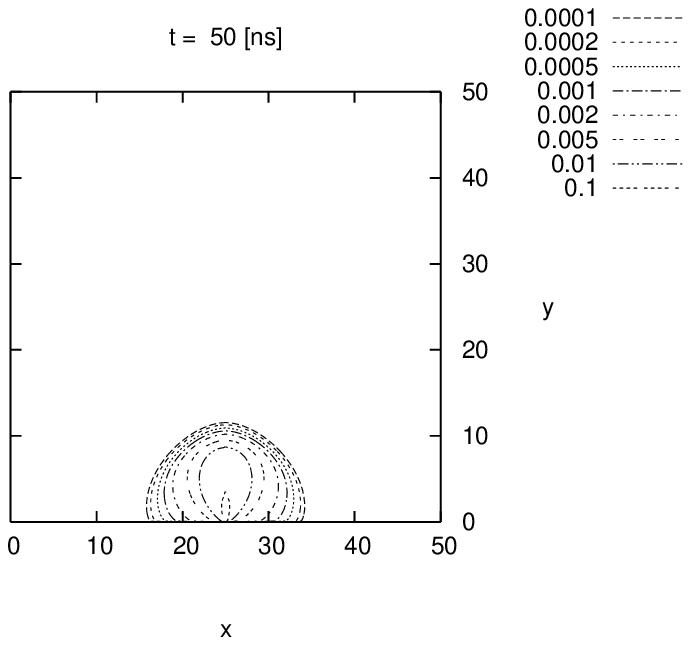}}
\subfigure[$t=100 \quad (k=1000)$]{\includegraphics[width=.5\textwidth]{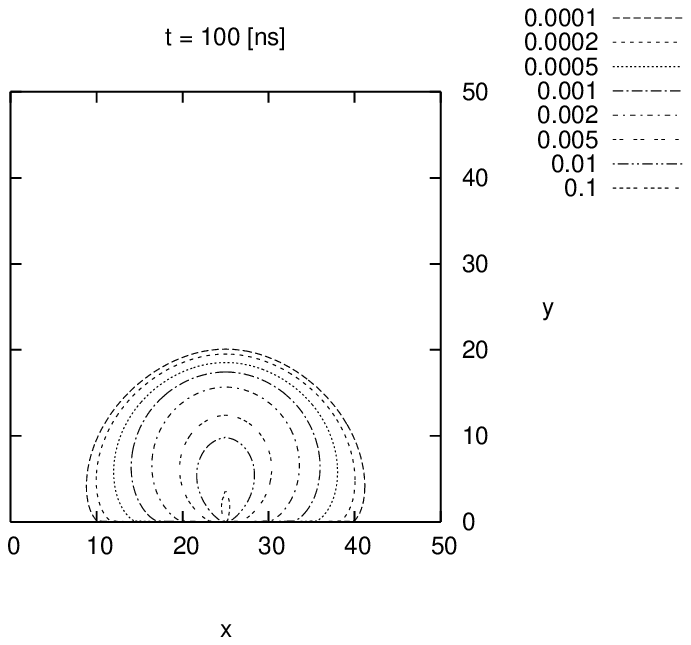}}
\subfigure[$t=200 \quad (k=2000)$]{\includegraphics[width=.5\textwidth]{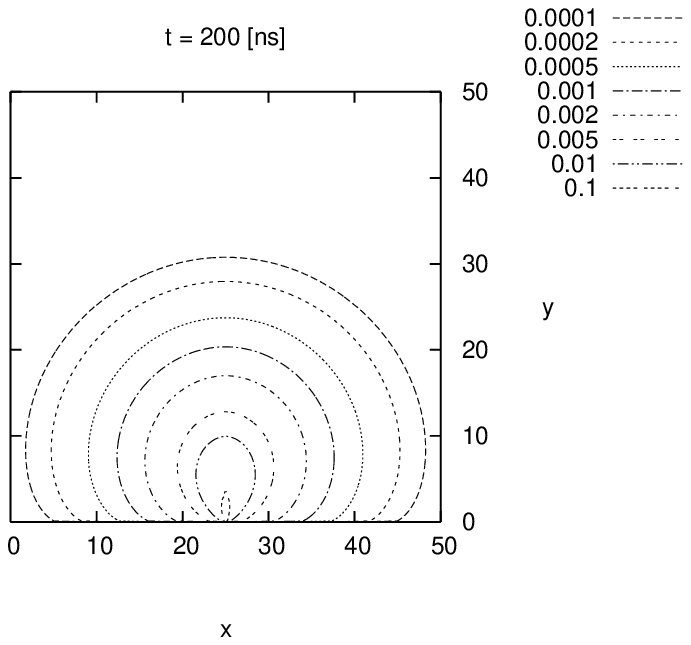}}
\subfigure[$t=400 \quad (k=4000)$]{\includegraphics[width=.5\textwidth]{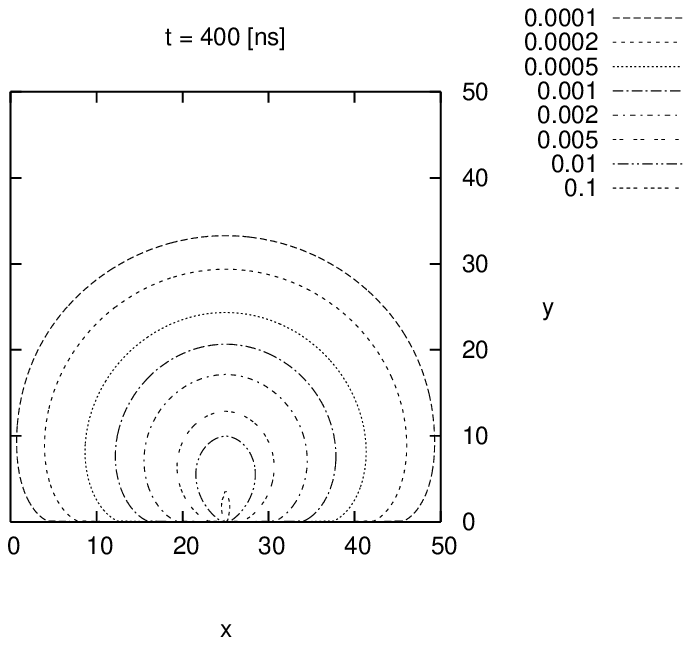}}
\caption{\label{fig:contourlightintensity}Contour of Integrated Light Intensity}
\end{figure}

As we have seen in \S \ref{sect:stat}, the sequence of grid 
functions $\{I^k\}_{k=0}^{\infty}$ generated by the scheme \eqref{2.2} 
converges to the solution $J_\Delta$ of \eqref{8.2} as $k$ 
tends to infinity. The convergence of $\{I^k\}$ to $J_\Delta$ 
is shown in Fig.~\ref{fig:convergence}, in which the horizontal axis is 
the number of time step $k$ and the vertical axis is the difference in 
logarithmic scale. The cross signs $(\times)$ show $\| I^k - I^{k-1} 
\|_\infty$ and the plus signs $(+)$ show $\| I^k - J_\Delta \|_\infty$. 
Both are approximately proportional to $0.99807^k$, which means that 
$\{I^k\}$ converges to $J_\Delta$ exponentially fast. 
The difference $\norm{I^k-J_\Delta}_\infty$ is saturated around $10^{-12}$ 
since the solution $J_\Delta$ of \eqref{8.2} 
is found by an iterative method with tolerance $10^{-12}$. 
On the other hand, $\norm{I^k-I^{k-1}}_\infty$ is not less than $10^{-16}$ 
because of rounding errors in IEEE 754 double precision. 
The conclusion of Theorem \ref{thm:4.1} is employed to determine the time step 
at which we stop the calculation of the scheme \eqref{2.2}. 

\begin{figure}[tbp]
\begin{center}
\includegraphics[width=.8\textwidth]{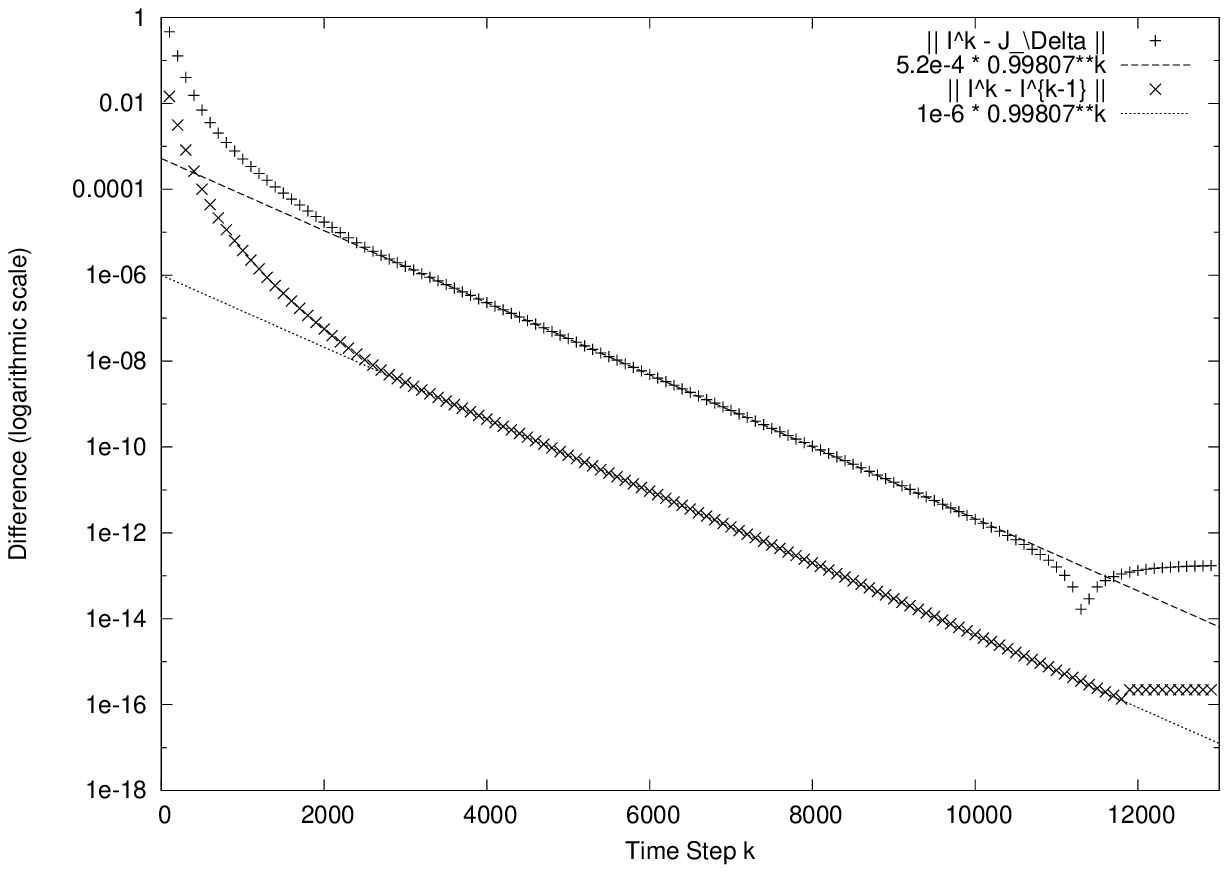}
\end{center}
\caption{\label{fig:convergence}Difference between Numerical Solutions and Steady State}
\end{figure}

\section{Three-dimensional case}\label{sect:3D}

We conclude the paper by a brief account of our numerical scheme in the 
three-dimensional case. 
Let $\Omega=(0,L_1)\times(0,L_2)\times(0,L_3)$ be a rectangular domain in $\R^3$, and 
put $X=\Omega\times S^2$. 
Let $M_1$, $M_2$, $M_3$, $M_\theta$, and $M_\phi$ be positive integers, and 
$\Delta t$ be a positive number. We consider the grid points $(t_k,x_{ijl},\xi_{mn})$ given by 
\begin{align*}
\Delta x_i&=L_i/M_i\quad (i=1,2,3), \quad 
\Delta\theta={\pi}/{M_\theta}, \quad \Delta\phi={2\pi}/{M_\phi}, \\
t_k&=k\Delta t, \quad 
x_{ijl}=(i\Delta x_1,j\Delta x_2,l\Delta x_3), \quad 
\theta_m=m\Delta\theta, \quad \phi_n=n\Delta\phi, \\
\xi_{mn}&=(\xi_{mn,1},\,\xi_{mn,2},\,\xi_{mn,3})=
(\sin\theta_m\cos\phi_n,\,\sin\theta_m\sin\phi_n,\,\cos\theta_m), 
\end{align*}
where $k,i,j,l,m,n\in\Z$. Using $I^k_{i,j,l,m,n}$ as the value 
corresponding to $I(t_k,x_{ijl},\xi_{mn})$, we discretize the problem 
\eqref{2.1} for $d=3$ as follows: 
\begin{subequations}\label{3d-scheme}
\begin{align}
\frac{1}{c(x_{ijl})}\frac{I^{k+1}_{i,j,l,m,n}-I^{k}_{i,j,l,m,n}}{\Delta t}
&=A_\Delta I^k_{i,j,l,m,n}-\Sigma_\Delta I^{k+1}_{i,j,l,m,n}+K_\Delta I^k_{i,j,l,m,n}+q^k_{i,j,l,m,n},\nonumber\\
&\,\qquad\qquad\qquad\qquad\qquad 0\leq k\leq T/\Delta t-1,\ (x_{ijl},\xi_{mn})\in X,\\
I^0_{i,j,l,m,n}&=I_0(x_{ijl},\xi_{mn}),
\quad\qquad (x_{ijl},\xi_{mn})\in X,\\
I^k_{i,j,l,m,n}&=I_1(t_k,x_{ijl},\xi_{mn}),
\qquad 0\leq k\leq T/\Delta t,\,\ (x_{ijl},\xi_{mn})\in\Gamma_{-},
\end{align}
\end{subequations}
where $q^k_{i,j,l,m,n}=q(t_k,x_{ijl},\xi_{mn})$, and the operators $A_\Delta$, $\Sigma_\Delta$, and $K_\Delta$ are defined by 
\begin{align*}
&A_\Delta I^k_{i,j,l,m,n}=\\
&\quad-\xi_{mn,1}\frac{I^k_{i+1,j,l,m,n}-I^k_{i-1,j,l,m,n}}{2\Delta x_1}
  +\abs{\xi_{mn,1}}\frac{I^k_{i+1,j,l,m,n}-2I^k_{i,j,l,m,n}+I^k_{i-1,j,l,m,n}}{2\Delta x_1}\\
&\quad-\xi_{mn,2}\frac{I^k_{i,j+1,l,m,n}-I^k_{i,j-1,l,m,n}}{2\Delta x_2}
  +\abs{\xi_{mn,2}}\frac{I^k_{i,j+1,l,m,n}-2I^k_{i,j,l,m,n}+I^k_{i,j-1,l,m,n}}{2\Delta x_2}\\
&\quad-\xi_{mn,3}\frac{I^k_{i,j,l+1,m,n}-I^k_{i,j,l-1,m,n}}{2\Delta x_3}
  +\abs{\xi_{mn,3}}\frac{I^k_{i,j,l+1,m,n}-2I^k_{i,j,l,m,n}+I^k_{i,j,l-1,m,n}}{2\Delta x_3},\\
&\Sigma_\Delta I^k_{i,j,l,m,n}=\{\mus(x_{ijl})+\mua(x_{ijl})\}I^k_{i,j,l,m,n},\\
&K_\Delta I^k_{i,j,l,m,n}=\mus(x_{ijl})\Delta\theta\Delta\phi
\sum_{\mu=1}^{M_\theta-1}\sum_{\nu=0}^{M_\phi-1}p(x_{ijl};\xi_{mn},\xi_{\mu\nu})\sin\theta_\mu
I^k_{i,j,l,\mu,\nu}.
\end{align*}

Sufficient conditions for the stability and convergence are given as follows. 

\begin{thm}
\label{thm:maximum} 
{\bf (Stability and positivity for $\boldsymbol{d=3}$)} 
Suppose that the discretization parameters satisfy 
\begin{subequations}
\begin{gather}
\dfrac{c^+\Delta t}{\Delta x_1}+\dfrac{c^+\Delta t}{\Delta x_2}+\dfrac{c^+\Delta t}{\Delta x_3}
\leq 1,\\
\Delta\theta^2\sup
\Abs{\frac{\partial^2}{\partial\theta^2}p\bigl(x;\xi,\xi'(\theta,\phi)\bigr)\sin\theta}
+\frac{\Delta\phi^2}{2}\sup
\Abs{\frac{\partial^2}{\partial\phi^2}p\bigl(x;\xi,\xi'(\theta,\phi)\bigr)\sin\theta}
\leq\frac{6}{\pi^2}\mu^*,
\label{3.3}
\end{gather}
\end{subequations}
where we put 
\[
\xi'(\theta,\phi)=(\sin\theta\cos\phi,\,\sin\theta\sin\phi,\,\cos\theta)
\]
and the supremum is taken over $(x,\xi)\in X,\ 0\leq\theta\leq\pi,\ 0\leq\phi\leq 2\pi$ 
in \eqref{3.3}. Then we have 
\[
\norm{I^k}_\infty\leq\norm{I_0}_\infty+\norm{I_1}_\infty+c^+\norm{q}_{\infty}T,\qquad 
0\leq k\leq T/\Delta t.
\]
Moreover if $q\geq 0$, $I_0\geq 0$, and $I_1\geq 0$, we have $I^k_{i,j,l,m,n}\geq 0$ for all $k,i,j,l,m,n$. 
\end{thm}

\begin{proof}
We can show that the solution of \eqref{3d-scheme} satisfies 
\begin{equation}
\label{replace-yyy}
\{1+c\Delta t(\mus+\mua)\}\Abs{I^{k+1}_{i,j,l,m,n}}
\leq\norm{I^k}_{\infty}+c\Delta t\Abs{K_\Delta I^{k}_{i,j,l,m,n}}+c\Delta t\norm{q}_{\infty}
\end{equation}
by repeating the argument which lead to \eqref{yyy} in Theorem \ref{thm:3.1}. 
Next we shall show that 
\begin{equation}
\label{replace-zzz}
\Delta\theta\Delta\phi
\sum_{\mu=1}^{M_\theta-1}\sum_{\nu=0}^{M_\phi-1}p(x_{ijl};\xi_{mn},\xi_{\mu\nu})\sin\theta_\mu
\leq 1+\mu^*,\quad (x_{ijl},\xi_{mn})\in X,
\end{equation}
which corresponds to \eqref{zzz}. 
The proof proceeds as in Theorem \ref{thm:3.1} 
by replacing \eqref{yyy} and \eqref{zzz} with \eqref{replace-yyy} and 
\eqref{replace-zzz}, respectively. 

To complete the proof we derive \eqref{replace-zzz}. 
Fix $(x_{ijl},\xi_{mn})\in X$ and consider the functions 
\begin{align}
g(\theta,\phi)&:=p\Bigl(x_{ijl};\xi_{mn},\xi'(\theta,\phi)\Bigr)\label{g=p},
\\
G(\theta)&:=\int_{0}^{2\pi}g(\theta,\phi)\,d\phi - 
\Delta\phi\sum_{\nu=0}^{M_\phi-1}g(\theta,\phi_\nu).
\label{G(theta)}
\end{align}
By the mean value theorem for integrals, 
there exists a point $\theta'\in(0,\pi)$ such that 
\begin{equation}
\label{int-mean}
\int_{0}^{\pi}G(\theta)\sin\theta\,d\theta=\pi G(\theta')\sin\theta'. 
\end{equation}
Applying Lemma \ref{lem:trapezoid} to the function $f(\phi)=g(\theta',\phi)$, 
we find a point $\phi'\in[0,2\pi)$ such that 
\begin{equation}
\label{G(theta')}
G(\theta')
=\frac{\pi}{12}\Delta\phi^2g_{\phi\phi}(\theta',\phi'). 
\end{equation}
We put \eqref{G(theta)} and \eqref{G(theta')} into \eqref{int-mean} to see that 
\begin{align*}
\int_{0}^{\pi}\int_{0}^{2\pi}g(\theta,\phi)\sin\theta\,d\phi d\theta
-\Delta\phi\sum_{\nu=0}^{M_\phi-1}\int_{0}^{\pi}g(\theta,\phi_\nu)\sin\theta \,d\theta
=\frac{\pi^2}{12}\Delta\phi^2 g_{\phi\phi}(\theta',\phi')\sin\theta'. 
\end{align*}
By the standard error estimate for the composite trapezoidal rule 
(see, e.g., \cite{Kress} Theorem 12.1), we have 
\begin{align*}
\int_{0}^{\pi}g(\theta,\phi_\nu)\sin\theta \,d\theta
=\Delta\theta\sum_{\mu=1}^{M_\theta-1}g(\theta_\mu,\phi_\nu)\sin\theta_\mu
-\frac{\pi}{12}\Delta\theta^2
\left.\frac{\partial^2}{\partial\theta^2}g(\theta,\phi_\nu)\sin\theta\right\vert_{\theta=\Theta_{\nu}}
\end{align*}
for some $\Theta_{\nu}\in(0,\pi)$. Thus 
\begin{align}
&\int_{0}^{\pi}\int_{0}^{2\pi}g(\theta,\phi)\sin\theta\,d\phi d\theta
-\Delta\phi\Delta\theta\sum_{\nu=0}^{M_\phi-1}
\sum_{\mu=1}^{M_\theta-1}g(\theta_\mu,\phi_\nu)\sin\theta_\mu\nonumber\\
&\qquad\qquad=\frac{\pi^2}{12}\Delta\phi^2 g_{\phi\phi}(\theta',\phi')\sin\theta'
-\frac{\pi}{12}\Delta\theta^2\Delta\phi\sum_{\nu=0}^{M_\phi-1}
\left.\frac{\partial^2}{\partial\theta^2}
g(\theta,\phi_\nu)\sin\theta\right\vert_{\theta=\Theta_{\nu}}.\label{aa}
\end{align}
Since $\dfrac{\partial^2}{\partial\theta^2}g(\theta,\phi)\sin\theta$ is continuous on 
$0\leq\theta\leq\pi,\,0\leq\phi\leq 2\pi$, the mean value theorem shows that 
there exists $(\theta'',\phi'')\in[0,\pi)\times[0,2\pi)$ such that 
\begin{equation}
\label{bb}
\Delta\phi\sum_{\nu=0}^{M_\phi-1}
\left.\frac{\partial^2}{\partial\theta^2}
g(\theta,\phi_\nu)\sin\theta\right\vert_{\theta=\Theta_{\nu}}
=2\pi\left.\frac{\partial^2}{\partial\theta^2}g(\theta,\phi'')\sin\theta
\right\vert_{\theta=\theta''}. 
\end{equation}
Noticing that 
\[
\int_{0}^{\pi}\int_{0}^{2\pi}g(\theta,\phi)\sin\theta \,d\phi \,d\theta
=\int_{S^2}p(x_{ijl};\xi_{mn},\xi')\,d\omega_{\xi'}=1,
\]
we obtain 
\[
\Delta\theta\Delta\phi\sum_{\mu=1}^{M_\theta-1}\sum_{\nu=0}^{M_\phi-1}g(\theta_\mu,\phi_\nu)\sin\theta_\mu
\leq 1+\mu^*
\]
from \eqref{aa}, \eqref{bb}, and \eqref{3.3}. 
\end{proof}

\begin{thm}
\label{thm:conv}
{\bf (Convergence for $\boldsymbol{d=3}$)} 
Suppose besides the hypotheses of Theorem \ref{thm:maximum} that 
(i) the mesh ratios $\Delta t/\Delta x_i$ ($i=1,2,3$) are kept fixed; 
(ii) the coefficients $\mua$ and $\mus$ are bounded; 
and (iii) the solution $I$ of the problem 
\eqref{2.1} belongs to $C^2\bigl([0,T)\times(X\cup\Gamma_{-})\bigr)$ 
and is bounded together with its derivatives of the first and the second order. 
Then we have 
\[
\|I(t_k,\cdot,\cdot)-I^k\|_\infty\leq C(\Delta t + \Delta\theta^2 +\Delta\phi^2),
\quad 0\leq k\leq T/\Delta t,
\]
where $C$ is a positive number independent of $\Delta t$, $\Delta\theta$, and $\Delta\phi$. 
\end{thm}

The proof is completely similar to that of Theorem \ref{thm:3.3} and is omitted.

\paragraph{Acknowledgement}

The authors wish to thank Professor Yuusuke Iso (Kyoto University) 
for advice and valuable comments. 
This study is partially supported by Grant-in-Aid for Challenging Exploratory Research 
No.~23654034. 
The second author's research is supported by Grant-in-Aid for Young Scientists (B) 
No.~23740075.


\end{document}